\documentclass[11pt, oneside]{article}   	
\usepackage{geometry}                		
\usepackage{graphicx}				
\usepackage{amssymb}
\usepackage{amsmath}
\usepackage{amsthm}

\newtheorem{theorem}[equation]{Theorem}
\newtheorem{thm}[equation]{Theorem}

\newtheorem{lem}[equation]{Lemma}
\newtheorem{cor}[equation]{Corollary}

\newtheorem{rem}[equation]{Remark}

\def\H{\mathbb H}

\title{Examples of aspherical Alexandrov spaces}
\author{Michael Kapovich}

\begin{document}
\maketitle

\begin{abstract}
We construct examples of 3-dimensional compact aspherical Alexandrov spaces without boundary which are not 
topological manifolds. 
\end{abstract}

The goal of this note (prompted by \cite{471666}) is to prove:

\begin{theorem}\label{thm:main}
There exist  3-dimensional compact aspherical Alexandrov spaces without boundary which are not topological manifolds. 
\end{theorem}

This disproves Conjecture 5.1 in \cite{survey}.

\medskip
We will construct these spaces as quotients of closed hyperbolic 3-manifolds by isometric involutions with isolated fixed points (and nonempty fixed point sets). These quotient spaces have isolated conical singularities which are cones over projective planes. The projection of the hyperbolic path-metric to the quotient gives this quotient the structure of an Alexandrov space (see e.g. \cite{AKP}).  
With a bit more work, this construction extends in all higher dimensions. Of course, not every space obtained this way will be aspherical, but some will. 

We will start with a compact connected orientable 3-manifold $M$ with nonempty boundary, which is aspherical, irreducible, boundary irreducible, atoroidal and acylindrical. We will also assume that each boundary component has genus $\ge 2$. By Thurston's theorem, $M$ admits structure of a convex hyperbolic manifold with totally geodesic boundary. One can also construct such manifolds using compact arithmetic hyperbolic manifolds of simplest type: After passage to a finite cover, every such manifold contains an embedded nonseparating totally geodesic hypersurface, see \cite{Millson}. Cutting along such a hypersurface results in a compact hyperbolic manifold with two  totally geodesic boundary components.

Let $\partial_i M, i\in I$, denote the boundary components of $M$. (The reader can assume that $\partial M$ is connected.) The covering of $M$ with pull-back hyperbolic metric  is isometric to a closed convex subset of $\H^3$ with nonempty totally geodesic boundary. 
 
Let $\tau: \partial M\to \partial M$ denote an orientation-preserving involution, such that $\tau$ preserves each boundary component and has nonempty fixed point set (which is necessarily finite). 
The quotient $M/\tau$ has a natural orbifold structure $\mathcal O$ with nonempty singular set and empty boundary, such that singularities are cones over projective planes. Another way to describe $\mathcal O$ is as follows. We define the {\em double of 
$M$ with respect to $\tau$}, $D_\tau M$ by
$$
D_\tau M= (M\times \{0,1\})/\sim, (x,0)\sim (x,1), x\in \partial M. 
$$
The involution $\tau$ extends to an involution $\sigma$ of $D_\tau M$, swapping the components of the double and 
preserving the boundary of $M$:
$$
\sigma(x,0)=(x,1). 
$$
 The double $D_\tau M$ is an orientable, closed, Haken and atoroidal, hence, by Thurston's theorem, admits a hyperbolic metric. Thus, by the Mostow Rigidity Theorem, $\sigma$ is homotopic to an isometry of $D_\tau M$;  by Waldhausen's theorem, $\sigma$ is also isotopic to this isometry. The involution $\sigma$ has finite nonempty fixed point set.   By the construction, the quotient space $D_\tau M/\sigma$ is not a topological manifold: Projections of fixed points of $\sigma$ are the singularities. 

\begin{rem}
Note that $\partial M$ need not be isotopic to a totally geodesic submanifold in $D_\tau M$. For this to be the case, we would need $\tau$ to be isotopic to an isometry of $\partial M$, which need not be the case. However, there are many examples where $\tau$ is an isometry,   coming from arithmetic hyperbolic manifolds. 
\end{rem}

 In general, the underlying topological space $|\mathcal O|=D_\tau M/\sigma$ of $\mathcal O$  need not be aspherical. We will prove that there exists a finite regular covering $M'\to M$ and an involution of $\partial M'$ such that applying the above construction to $M'$ we obtain a new orbifold $\mathcal O'$ whose underlying space is aspherical and which is not a manifold.

We let $p: \hat M\to M$ denote the regular covering corresponding to the normal closure $N$ in $\pi_1(M)$ of the collection of peripheral subgroups  $\pi_1(\partial_i M), i\in I$. We will refer to this as the {\em peripheral covering} of $M$. 
Then the manifold $\hat M$ has totally geodesic boundary and the covering map $p: \hat M\to M$ restricts to an isometric homeomorphism from each boundary component $\partial_j \hat M, j\in J$, to appropriate component of $\partial M$. In general, very little can be said about the subgroup $N$ and the covering space $\hat M$. However, it turns out that after applying this construction to a sufficiently deep finite index subgroup of $\pi_1(M)$, the group $N$ splits as a free product of peripheral subgroups. For a finite index normal subgroup $\Gamma'< \pi_1(M)$ and the corresponding covering $M'\to M$, we again use the notation $p': \hat M'\to M'$ for the corresponding peripheral covering and $N'$ the corresponding normal closure in $\pi_1(M')$ of the collection of peripheral subgroups.

We will need the  following consequence of \cite{DGO} and \cite[Theorem 2.5]{Sun}:

\begin{thm}
There exists a finite regular covering $M'\to M$ such that the normal subgroup $N'< \pi_1(M')$ splits as a free product of subgroups $\pi_1(\partial_j \hat M')$, where $j\in J'$ index boundary components of $\hat M'$. 
\end{thm}

In order to apply the results of \cite{DGO} and \cite{Sun} we note that $\pi_1(M)$ is residually finite and that 
$(\pi_1(M), \{\pi_1(\partial_i M): i\in I\})$ is a relatively hyperbolic group, since $M$ has geodesic boundary. (The latter ensures that each peripheral subgroup of $\pi_1(M)$ is malnormal and conjugates of distinct peripheral subgroups have trivial intersections.) 
 
Now, pick an interior  point $x\in \hat M'$ and connect $x$ to each boundary component by a simple arc, so that these arcs intersect only at $x$. This results in an infinite graph $Y\subset \hat M'$. Let $X$ denote the union of this graph and $\partial \hat M'$. The CW complex $X$ is connected, its fundamental group is naturally a free product of the groups $\pi_1(\partial_j \hat M'), j\in J$. Hence, the inclusion map 
$X\to \hat M'$ induces an isomorphism of fundamental groups and is, therefore, a homotopy-equivalence (since both spaces are aspherical). Moreover, we obtain a homotopy-equivalence of pairs
$$
(X, \partial \hat M')\to (\hat M', \partial \hat M'). 
$$

We now choose an involution $\tau: \partial M'\to \partial M'$ as above, such that each component of $\partial M'/\tau$ 
has genus $\ge 1$.  Let $\hat \tau: \partial \hat M'\to \partial \hat M'$ be a lift of the involution $\tau$ such that $\hat\tau$ preserves each boundary component. Then taking the quotient orbifolds ${\mathcal O}'= M'/\tau$ and $\hat{\mathcal O}'= \hat M'/\hat \tau$, the covering map $p': \hat M'\to M'$ descends to an orbi-covering map $q': \hat{\mathcal O}'\to \mathcal O'$. By the construction, 
this orbi-covering is an ordinary covering of the underlying spaces of these orbifolds. Hence, 
$|\hat{\mathcal O}'|$ is aspherical if and only if $|\mathcal{O}'|$ is. 

\medskip 
We let $Z$ denote the quotient of $X$ by the involution $\hat\tau$.

\begin{lem}
$Z$ is aspherical. 
\end{lem}
\begin{proof} $Z$ is homotopy-equivalent to a wedge of copies of components of  $\partial \hat M'/\tau$. By our assumption, each of these quotient surfaces has genus $\ge 1$, hence, is aspherical. Thus, $Z$ is aspherical as well.   
\end{proof}

The homotopy-equivalence $X\to  \hat M'$ descends to a homotopy-equivalence $Z\to |\hat{\mathcal O}'|$ (since it defined a homotopy-equivalence of pairs). 

\begin{cor}
$|\hat{\mathcal O}'|$ and, hence, $|\mathcal{O}'|$ is aspherical. 
\end{cor}

By the construction, $|\mathcal{O}'|$ has structure of an Alexandrov space (since it is a good orbifold) with empty boundary. It is not a manifold since it has conical singularities.  This proves Theorem \ref{thm:main}. 

\medskip
{\bf Acknowledgements.} I am grateful to Denis Osin for the references \cite{DGO} and \cite{Sun}. 

\newpage

\bibliographystyle{amsalpha}
\bibliography{references}

\providecommand{\bysame}{\leavevmode\hbox to3em{\hrulefill}\thinspace}
\providecommand{\MR}{\relax\ifhmode\unskip\space\fi MR }
\providecommand{\MRhref}[2]{%
  \href{http://www.ams.org/mathscinet-getitem?mr=#1}{#2}
}
\providecommand{\href}[2]{#2}
\begin{thebibliography}{GGNZ22}

\bibitem[AKP24]{AKP}
Stephanie Alexander, Vitali Kapovitch, and Anton Petrunin, \emph{Alexandrov
  geometry---foundations}, Graduate Studies in Mathematics, vol. 236, American
  Mathematical Society, Providence, RI, [2024] \copyright 2024. \MR{4734965}

\bibitem[DGO17]{DGO}
F.~Dahmani, V.~Guirardel, and D.~Osin, \emph{Hyperbolically embedded subgroups
  and rotating families in groups acting on hyperbolic spaces}, Mem. Amer.
  Math. Soc. \textbf{245} (2017), no.~1156, v+152. \MR{3589159}

\bibitem[GGNZ22]{survey}
Fernando Galaz-Garc{\'i}a and Jes{\'u}s N{\'u}{\~{n}}ez-Zimbr{\'o}n,
  \emph{Three-dimensional {A}lexandrov spaces: A survey}, Recent Advances in
  Alexandrov Geometry (Cham) (Gerardo Arizmendi~Echegaray, Luis
  Hern{\'a}ndez-Lamoneda, and Rafael Herrera~Guzm{\'a}n, eds.), Springer
  International Publishing, 2022, pp.~49--88.

\bibitem[Kat]{471666}
Katrina, \emph{Example of an aspherical non-manifold {A}lexandrov 3-space},
  MathOverflow, URL:https://mathoverflow.net/q/471666.

\bibitem[Mil76]{Millson}
John~J. Millson, \emph{On the first {B}etti number of a constant negatively
  curved manifold}, Ann. of Math. (2) \textbf{104} (1976), no.~2, 235--247.
  \MR{422501}

\bibitem[Sun20]{Sun}
B.~Sun, \emph{Cohomology of group theoretic {D}ehn fillings {I}:
  {C}ohen-{L}yndon type theorems}, J. Algebra \textbf{542} (2020), 277--307.
  \MR{4019787}

\end{thebibliography}

\noindent Department of Mathematics\\1 Shields Ave.\\
  University of California, Davis\\
  CA 95616, USA
  
\noindent kapovich@math.ucdavis.edu

\end{document}